\documentclass[11pt]{article}
\pdfoutput=1
\usepackage[UKenglish]{isodate}
\usepackage[T1]{fontenc}
\usepackage[noBBpl]{mathpazo}

\usepackage{amsmath, amssymb, amsfonts, amsthm, dsfont, zlmtt, array, url}

\usepackage[a4paper,left=2.8cm,right=2.8cm,top=2.8cm,bottom=2.8cm]{geometry}
\usepackage{microtype}
\usepackage[shortcuts]{extdash}
\cleanlookdateon

\usepackage{titlesec}
\titlespacing*{\section}{0pt}{\baselineskip}{0pt}
\titlespacing*{\subsection}{0pt}{0.66\baselineskip}{0pt}

\usepackage[shortlabels]{enumitem}
\setlist{leftmargin=0.8cm,topsep=0pt,itemsep=-2pt}
\setlist[enumerate]{label=\rm{(\roman*)}}

\numberwithin{equation}{section}

\makeatletter
\g@addto@macro\normalsize{%
  \setlength\abovedisplayskip{0.4\baselineskip plus 0.4\baselineskip}
  \setlength\belowdisplayskip{0.4\baselineskip plus 0.4\baselineskip}
  \setlength\abovedisplayshortskip{-0.3\baselineskip}
  \setlength\belowdisplayshortskip{0.4\baselineskip plus 0.4\baselineskip}
}
\makeatother

\makeatletter\def\blfootnote{\gdef\@thefnmark{}\@footnotetext} \makeatother
\newcommand{\dateline}[1]{\enlargethispage{16pt}\blfootnote{\phantom{\Large M}\hspace{-1em}\hspace{-20pt}\emph{Date} #1}}

\renewenvironment{thebibliography}[1]
{ \begin{oldthebibliography}{#1}
  \setlength{\parskip}{0pt}
  \setlength{\itemsep}{1pt plus 0.3ex}
  \bgroup\footnotesize }
{ \egroup \end{oldthebibliography} }

\makeatletter
\renewenvironment{proof}[1][\proofname]{\par
  \pushQED{\qed}%
  \normalfont
  \topsep2pt \partopsep1pt 
  \trivlist
  \item[\hskip\labelsep
        \itshape
    #1\@addpunct{.}]\ignorespaces
}{%
  \popQED\endtrivlist\@endpefalse
  \addvspace{6pt plus 6pt}
}
\makeatother

\newtheoremstyle{shdefinition}{8pt}{4pt}{}{}{\bfseries\boldmath}{.}{0.3em}{} 
\newtheoremstyle{shplain}{8pt}{4pt}{\itshape}{}{\bfseries\boldmath}{.}{0.3em}{} 
\theoremstyle{shdefinition}
\newtheorem{definition}{Definition}[section]
\newtheorem{shdefinition}{Definition}
\newtheorem{example}[definition]{Example}
\newtheorem*{acknowledgements*}{Acknowledgements}
\theoremstyle{shplain}
\newtheorem{theorem}[definition]{Theorem}
\newtheorem{shtheorem}[shdefinition]{Theorem}

\newtheorem{shcorollary}[shdefinition]{Corollary}
\newtheorem{proposition}[definition]{Proposition}
\newtheorem{lemma}[definition]{Lemma}

\newcommand{\e}{\varepsilon}
\newcommand{\p}{\varphi}
\renewcommand{\l}{\lambda}
\newcommand{\C}{\mathcal{C}}
\newcommand{\M}{\mathcal{M}}
\renewcommand{\S}{\mathcal{S}}
\newcommand{\<}{\langle}
\renewcommand{\>}{\rangle}
\renewcommand{\leq}{\leqslant}
\renewcommand{\geq}{\geqslant}
\newcommand{\soc}{\mathrm{soc}}
\newcommand{\Aut}{\mathrm{Aut}}
\newcommand{\Out}{\mathrm{Out}}
\newcommand{\Inndiag}{\mathrm{Inndiag}}
\newcommand{\F}{\mathbb{F}}

\renewcommand{\div}{\mid}
\newcommand{\PSL}{\mathrm{PSL}}
\newcommand{\PSU}{\mathrm{PSU}}

\begin{document}

\begin{center} 
{\LARGE \textbf{The maximal size of a minimal generating set}} \\[11pt]
{\Large Scott Harper}                                          \\[22pt]
\end{center}

\begin{center}
\begin{minipage}{0.8\textwidth}
\small A generating set for a finite group $G$ is \emph{minimal} if no proper subset generates~$G$, and $m(G)$ denotes the maximal size of a minimal generating set for $G$. We prove a conjecture of Lucchini, Moscatiello and Spiga by showing that there exist $a,b > 0$ such that any finite group $G$ satisfies $m(G) \leq a \cdot \delta(G)^b$, for $\delta(G) = \sum_{\text{$p$ prime}} m(G_p)$ where $G_p$ is a Sylow $p$-subgroup of $G$. To do this, we first bound $m(G)$ for all almost simple groups of Lie type (until now, no nontrivial bounds were known except for groups of rank $1$ or $2$). In particular, we prove that there exist $a,b > 0$ such that any finite simple group $G$ of Lie type of rank $r$ over the field $\mathbb{F}_{p^f}$ satisfies $r + \omega(f) \leq m(G) \leq a(r + \omega(f))^b$, where $\omega(f)$ denotes the number of distinct prime divisors of $f$. In the process, we confirm a conjecture of Gill and Liebeck that there exist $a,b > 0$ such that a minimal base for a faithful primitive action of an almost simple group of Lie type of rank $r$ over $\mathbb{F}_{p^f}$ has size at most $ar^b + \omega(f)$.
\par
\end{minipage}
\end{center}

\dateline{28 June 2023 \ \emph{2020 Mathematics Subject Classification} Primary: 20F05, Secondary: 20B05, 20E28, 20E32.}
 
\section{Introduction} \label{s:intro}

Since a generating set for a group remains a generating set if additional elements are added, it is natural to focus on generating sets for which no proper subset is a generating set; we call these \emph{minimal generating sets} (they are also known as \emph{independent generating sets}). A minimal generating set need not have minimum possible size. For instance, if $n \geq 4$, then $\{ (1 \ 2), (2 \ 3), \dots, (n-1 \ n) \}$ is a minimal generating set for $S_n$, but it has size $n-1$, which exceeds the minimum size possible size of $2$. How large can a minimal generating set be? 

Let $G$ be a finite group. We begin by comparing what is known about the minimum size $d(G)$ and maximum size $m(G)$ of a minimal generating set for $G$. If $G$ is a $p$-group, then $d(G) = m(G)$ (this follows from Burnside's basis theorem), so if $G$ is nilpotent, then $d(G) = \max_{\text{$p$ prime}} d(G_p)$ and $m(G) = \sum_{\text{$p$ prime}} d(G_p)$ where $G_p$ is a Sylow \mbox{$p$-subgroup} of $G$. In 1989, Guralnick \cite{ref:Guralnick89} and Lucchini \cite{ref:Lucchini89} independently proved that all finite groups $G$ satisfy $d(G) \leq \max_{\text{$p$ prime}} d(G_p) + 1$. Do all finite groups satisfy $m(G) \leq  \sum_{\text{$p$ prime}} d(G_p)+ 1$? Writing $\delta(G) = \sum_{\text{$p$ prime}} d(G_p)$, Lucchini, Moscatiello and Spiga \cite{ref:LucchiniMoscatielloSpiga21} showed that $m(G) \leq \delta(G)+1$ is false in general, but they conjectured that there exist $a,b > 0$ such that every finite group $G$ satisfies $m(G) \leq a \cdot \delta(G)^b$. Our first theorem confirms this local-to-global conjecture.

\begin{shtheorem} \label{thm:main}
There exist $a,b > 0$ such that if $G$ is any finite group, then $m(G) \leq a \cdot \delta(G)^b$. Moreover, this is true for $a = 10^{10}$ and $b=10$.
\end{shtheorem}

In \cite{ref:LucchiniMoscatielloSpiga21}, Theorem~\ref{thm:main} is reduced to a statement about almost simple groups. (A group $G$ is \emph{almost simple} if $G_0 \leq G \leq \Aut(G_0)$ for a nonabelian simple group $G_0$.) Even for finite simple groups $G$, while it has long been known that $d(G) \leq 2$ (see \cite{ref:AschbacherGuralnick84}), little is known about $m(G)$. In 2000, Whiston \cite{ref:Whiston00} (using the Classification of Finite Simple Groups) proved $m(A_n) = n-2$. Later, in 2002, Whiston and Saxl \cite{ref:SaxlWhiston02} proved that $m(\PSL_2(p^f)) \leq \max\{ 6, \omega(f) + 2 \}$ with equality if $\omega(f) \geq 4$, and the exact value of $m(\PSL_2(p))$ is given in \cite{ref:Jambor13}. (Throughout, $\omega(n)$ is the number of distinct prime divisors of $n$, and $\Omega(n)$ is the number of prime divisors of $n$ counted with multiplicity.) Except for the 3-dimensional classical groups studied in \cite{ref:Keen11}, no nontrivial bounds exist for any other finite simple group. This motivates our second theorem, which we will use to prove Theorem~\ref{thm:main}.\nopagebreak

\begin{shtheorem} \label{thm:almost_simple}
There exist $\alpha,\beta > 0$ such that if $G$ is an almost simple group of Lie type of rank $r$ over $\F_{p^f}$ (where $p$ is prime), then $m(G) \leq \alpha (r+\omega(f))^\beta$. Moreover, this is true for $\alpha=10^5$ and $\beta=10$.
\end{shtheorem}

Up to improving the values of $\alpha$ and $\beta$, Theorem~\ref{thm:almost_simple} is best possible since if $G$ is a finite simple group of Lie type of rank $r$ over $\F_{p^f}$, then $m(G) \geq r+\omega(f)$ (see Proposition~\ref{prop:lower}).

The invariant $m(G)$ also plays a role in the product replacement algorithm for producing random elements of $G$ (see \cite{ref:CellerLeedhamGreenMurrayNiemeyerOBrien95}). This algorithm involves a random walk on the product replacement graph $\Gamma_n(G)$, and Diaconis and Saloff-Coste \cite{ref:DiaconisSaloffCoste98} proved that for large enough $n$ this random walk reaches the uniform distribution in time $|G|^{O(m(G))}n^2\log{n}$.

To prove Theorem~\ref{thm:almost_simple}, we relate $m(G)$ to some well studied invariants in permutation group theory. Let $G$ be a finite group acting faithfully on a set $X$. A sequence $(x_1,\dots,x_k)$ of points in $X$ is a \emph{base} if the pointwise stabiliser $G_{(x_1, \dots, x_k)}$ is trivial. This subject has a long history, with many connections to abstract group theory, computational group theory and graph theory, see the survey \cite{ref:BaileyCameron11}. A base $(x_1, \dots, x_k)$ is \emph{irredundant} if we have a proper subgroup chain $G > G_{(x_1)} > G_{(x_1,x_2)} > \cdots > G_{(x_1,x_2,\dots,x_k)} = 1$, and it is \emph{minimal} if no proper subsequence of $(x_1, \dots, x_k)$ is a base. A minimal base is irredundant, but the converse need not hold. Let $I(G,X)$ and $B(G,X)$ be the maximum size of an irredundant and minimal base, respectively. 

Much is known about the minimum size of a base, for instance the resolutions of Pyber's conjecture \cite{ref:DuyanHalasiMaroti18} and Cameron's conjecture \cite{ref:BurnessLiebeckShalev09}, but less is known about $I(G,X)$ and $B(G,X)$. However, Gill and Liebeck \cite{ref:GillLiebeck23} recently proved that any almost simple group of Lie type of rank $r$ over $\F_{p^f}$ (where $p$ is prime) acting faithfully and primitively on $X$ satisfies the bound $B(G,X) \leq I(G,X) \leq 177r^8 + \Omega(f)$. As explained in \cite[Example~5.1]{ref:GillLiebeck23}, $I(G,X)$ must depend on $\Omega(f)$, but Gill and Liebeck conjecture that $B(G,X)$ should only depend on $\omega(f)$ (see \cite[Conjecture~5.2]{ref:GillLiebeck23} for a precise statement). Our final theorem proves this conjecture. 

We actually prove a stronger result (also conjectured in \cite{ref:GillLiebeck23}) that is more convenient for proving Theorem~\ref{thm:almost_simple}. For a finite group $G$ acting on $X$, a sequence $S$ in $X$ is \emph{independent} if $G_{(S')} > G_{(S)}$ for proper all subsequences $S'$ of $S$, and the \emph{height}, denoted $H(G,X)$, is the maximum size of an independent sequence. Note that $B(G,X) \leq H(G,X) \leq I(G,X)$.

\begin{shtheorem} \label{thm:height}
There exist $A,B > 0$ such that if $G$ is an almost simple group of Lie type of rank $r$ over $\F_{p^f}$ (where $p$ is prime) acting faithfully and primitively on a set $X$, then $H(G,X) \leq Ar^B+\omega(f)$. Moreover, this is true for $A = 177$ and $B=8$.
\end{shtheorem}

For a finite group $G$ acting on $X$, the height $H(G,X)$ is related to the \emph{relational complexity} $RC(G,X)$ via the inequality $RC(G,X) \leq H(G,X)+1$. Relational complexity arose in model theory and it has been the subject recent work, including Gill, Liebeck and Spiga's recent proof \cite{ref:GillLiebeckSpiga22} of Cherlin's conjecture that classifies the primitive actions satisfying $RC(G,X)=2$.

\begin{shcorollary} \label{cor:height}
If $G$ is an almost simple group of Lie type of rank $r$ over $\F_{p^f}$ (where $p$ is prime) acting faithfully and primitively on a set $X$, then the following hold
\begin{enumerate}
\item $B(G,X)  \leq 177r^8 + \omega(f)$
\item $RC(G,X) \leq 177r^8 + \omega(f) + 1$.
\end{enumerate}
\end{shcorollary}

\enlargethispage{11pt}

\begin{acknowledgements*}
The author thanks Andrea Lucchini for introducing him to this topic, Nick Gill and Martin Liebeck for information in advance of \cite{ref:GillLiebeck23}, Colva Roney-Dougal for several influential conversations, particularly concerning Section~\ref{ss:proof_height}, and the anonymous referee for useful comments. The author is a Leverhulme Early Career Fellow, and he thanks the Leverhulme Trust for their support.
\end{acknowledgements*}

\section{Preliminaries} \label{s:prelims}

\subsection{Maximal subgroups of almost simple groups} \label{ss:p_maximal}

The maximal subgroups of almost simple groups of Lie type are described by Theorem~\ref{thm:maximal}, which combines two theorems of Liebeck and Seitz \cite[Theorem~2]{ref:LiebeckSeitz90} and \cite[Theorem~2]{ref:LiebeckSeitz98}. 

Let $p$ be prime and let $X$ be a linear algebraic group over $\overline{\mathbb{F}}_p$, which from now on we call an \emph{algebraic group}. For a Steinberg endomorphism $\sigma$ of $X$, write $X_\sigma = \{ x \in X \mid x^\sigma = x \}$. A finite group $O^{p'}(X_\sigma)$ for a simple algebraic group $X$ of adjoint type and a Steinberg endomorphism $\sigma$ is usually simple, and in this case we call it a \emph{finite simple group of Lie type}. (Here $O^{p'}(G)$ is the subgroup generated by the $p$-elements of $G$.) In particular, for us, the Tits group ${}^2F_4(2)'$ is not a finite simple group of Lie type. Throughout, by \emph{rank} we mean untwisted rank.

\begin{theorem} \label{thm:maximal}
Let $G$ be an almost simple group of Lie type. Write $\soc(G) = O^{p'}(X_\sigma)$ for a simple algebraic group $X$ of adjoint type and a Steinberg endomorphism $\sigma$ of $X$. Let $M$ be a maximal subgroup of $G$ not containing $\soc(G)$. Then $M$ is one of the following
\begin{enumerate}[{\rm (I)}]
\item $N_G(Y_\sigma \cap \soc(G))$ for a maximal closed $\sigma$-stable positive-dimensional subgroup $Y$ of $X$
\item $N_G(X_\alpha \cap \soc(G))$ for a Steinberg endomorphism $\alpha$ of $X$ such that $\alpha^k=\sigma$ for a prime $k$
\item a local subgroup not in (I)
\item an almost simple group not in (I) or (II)
\item the Borovik subgroup: $M \cap \soc(G) = (A_5 \times A_6).2^2$ with $\soc(G) = E_8(q)$ and $p \geq 7$.
\end{enumerate}
\end{theorem}

We say that a core-free maximal subgroup of an almost simple group of Lie type has \emph{type} (I), (II), (III), (IV) or (V) if it arises in case (I), (II), (III), (IV) or (V) of Theorem~\ref{thm:maximal}, respectively.

In the remainder of this section, we collect together information about the subgroups appearing in cases (I)--(V) that we will use in the proof of Theorem~\ref{thm:almost_simple}.

\subsection{Aschbacher's theorem and type~(I*) subgroups} \label{ss:p_aschbacher}

One usually categorises the maximal subgroups of an almost simple classical group $G$ via Aschbacher's subgroup structure theorem \cite{ref:Aschbacher84}. Following the notation of  Kleidman and Liebeck in \cite{ref:KleidmanLiebeck}, one can define a geometric class of subgroups $\C = \C_1 \cup \dots \cup \C_8$ (see \cite[Chapter~3]{ref:KleidmanLiebeck}) and a class $\S$ of almost simple groups (see \cite[Section~1.2]{ref:KleidmanLiebeck}) such that for every subgroup $H$ of $G$ not containing $\soc(G)$ either $H \leq M$ for a maximal subgroup $M \in \C$ or else $H \in \S$. The subgroups in $\C$ are all of type (I), (II) or (III) (with (II) and (III) broadly overlapping with $\C_5$ and $\C_6$ subgroups, respectively). However, Case~(I) also includes almost simple groups in $\S$ whose socle is a group of Lie type in defining characteristic. Therefore, it is convenient to make the following definitions.

Let $G$ be an almost simple group of Lie type. If $G$ is classical, then let (I*) be the set of all maximal subgroups $M$ of $G$ that are in (I) and are in the geometric class $\C$ or are twisted tensor product subgroups (see \cite{ref:Schaffer99}), and define (IV*) as (IV) without the groups of Lie type in defining characteristic. (It is natural to include the twisted tensor product subgroups in (I*) as they arise as geometric subgroups of the ambient algebraic group, see \cite[Theorem~2]{ref:LiebeckSeitz98} and the remark that follows it.) If $G$ is exceptional, then let (I*) be (I) and (IV*) be (IV). 

\begin{proposition} \label{prop:maximal}
Let $G$ be an almost simple classical group defined over $\F_q$ and let $H \leq G$ not contain $\soc(G)$. Then either $H \leq M$ for a maximal subgroup $M$ of $G$ of type (I*), (II), (III) or (IV*), or $H$ is an almost simple group of Lie type defined over a subfield of $\F_q$ satisfying $\mathrm{rank}(H) < \mathrm{rank}(G)$.
\end{proposition}

\begin{proof}
Assume that $H$ does not lie in a maximal subgroup of type (I*), (II), (III) or (IV*). Let $q=p^f$ where $p$ is prime, and let $\F_q^n$ be the natural module for $G$. Since (I*), (II) and (III) cover the entire geometric class $\C$ of maximal subgroups of $G$, we know that $H$ is not a subgroup of any maximal subgroup $M$ contained in the geometric class $\C$. Therefore, by the main theorem of \cite{ref:Aschbacher84}, we deduce that $H$ is contained in $\S$. Since $H$ is not in (IV*), $H$ must be a group of Lie type defined over $\F_{p^e}$ for some $e$. By \cite[Proposition~5.4.6]{ref:KleidmanLiebeck} (which is an application of Steinberg's twisted tensor product theorem), $f$ divides $de$, for some $d \in \{1,2,3\}$, and the embedding of $H$ in $G$ affords an irreducible $\F_{p^e}H$-module $V$ of dimension $n^{f/de}$. Since $H$ is not contained in any maximal twisted tensor product subgroup, by \cite[Corollary~6]{ref:Seitz88}, we must have $f=de$, so $H$ is defined over $\F_{p^{f/d}}$ and $\dim{V} = n$. 

It remains to prove that $\mathrm{rank}(H) < \mathrm{rank}(G)$. Referring to the bounds on the dimension of the minimal module in \cite[Proposition~5.4.13]{ref:KleidmanLiebeck}, since $H$ has an irreducible module of dimension $n$ either $\mathrm{rank}(H) < \mathrm{rank}(G)$ or $(H,V)$ is one of a small number of possibilities all of which arise in the geometric class $\C$ or case (II). For instance, if $G = \mathrm{PSp}_n(q)$ with $n > 8$, then $n$ is strictly smaller than the dimension of the minimal module of any finite simple group of Lie type of $\mathrm{rank}(G) = n/2$ except $\PSL^\pm_{n/2+1}(q)$ or $\mathrm{P}\Omega^\pm_n(q)$, and, by \cite[Proposition~5.4.11]{ref:KleidmanLiebeck}, the former groups have no irreducible modules of dimension $n$ and latter groups embed in $G$ as $\C_8$ subgroups if at all. However, $H$ is not in $\C$ or (II), so $\mathrm{rank}(H) < \mathrm{rank}(G)$.
\end{proof}

\vspace{-7pt}

\enlargethispage{14pt}

\subsection{Length and type~(IV*) subgroups} \label{ss:p_length}

Let us introduce an invariant that we use throughout the paper. Let $G$ be a finite group. A \emph{subgroup chain} of $G$ of length $k$ is a sequence $G = G_0 > G_1 > \dots > G_k = 1$, and the \emph{length} of $G$, written $\ell(G)$, is the maximal length of a subgroup chain of $G$. To see the significance for this paper, note that if $X = \{ x_1, \dots, x_k \}$ is a minimal generating set for $G$, then 
\[
G = \<x_1,\dots,x_k\> > \<x_2,\dots,x_k\> > \cdots > \<x_k\> > 1
\]
is a subgroup chain, so $m(G) \leq \ell(G)$. There are many results on length, and we highlight one result that we will use later. Cameron, Solomon and Turull proved in \cite[Theorem~1]{ref:CameronSolomonTurull89} that 
\begin{equation}
\ell(S_n) = \lfloor \tfrac{3n-1}{2} \rfloor - b_n \label{eq:length}
\end{equation}
where $b_n$ is the number of ones in the base $2$ expansion of $n$.

\begin{proposition} \label{prop:length}
There exists $C > 0$ such that if $G$ is an almost simple group of Lie type of rank $r$ and $M$ is a type (IV*) maximal subgroup of $G$, then $\ell(M) \leq C r$. Moreover, this is true with $C = 192$.
\end{proposition}

\begin{proof}
First assume that $\soc(M)$ is sporadic. Then $\ell(\soc(M))$ is given in \cite[Tables~III \&~IV]{ref:CameronSolomonTurull89} (the ``probable values'' have since been verified), whence we deduce that $\ell(M) \leq 52$. 

Next assume that $\soc(M) = A_d$ for some $d \geq 5$. If $G$ is exceptional, then $d \leq 18$ by \cite[Theorem~8]{ref:LiebeckSeitz03}, so $\ell(M) \leq \ell(S_d) + 1 \leq 28$ by \eqref{eq:length}, and if $G$ is classical in dimension $n$, then \cite[Proposition~5.3.7]{ref:KleidmanLiebeck} implies that $d \leq 2n-1$, so $\ell(M) \leq 3n-1 \leq 6r+2$, again by \eqref{eq:length}.

Finally assume that $\soc(M)$ is a group of Lie type of rank $r_0$ over $\F_{q_0}$ where $q_0$ is a power of a prime $p_0$. If $G$ is classical in dimension $n$, then $p_0 \neq p$ and \cite[Theorem~5.3.9]{ref:KleidmanLiebeck} implies that $q_0^{r_0} \leq n^2$, so using the facts that $|M| \leq q_0^{12r_0^2}$ and $n \leq 4r$, we have
\[
\ell(M) \leq \log_2|M| \leq \log_2 (q_0^{12r_0^2}) \leq 12 (r_0 \log_2 q_0)^2 \leq 48 (\log_2 n)^2 \leq 48 (\log_2 4r)^2 \leq 192r.
\]
If $G$ is exceptional, then there are only finitely many possibilities for $M$. Consulting \cite[Tables~10.3 \&~10.4]{ref:LiebeckSeitz99} for the case $p_0 \neq p$ and \cite[Theorem~8]{ref:LiebeckSeitz03} for the case $p_0 = p$, we see that in all cases $|M| < 2^{200}$, so $\ell(M) \leq 200 \leq 100r$.
\end{proof}

\subsection{Enumerating maximal subgroups} \label{ss:p_count}

In this section, we prove the following result, which gives a bound on the number of maximal subgroups of various types in almost simple groups of Lie type.

\begin{proposition} \label{prop:count}
Let $G$ be an almost simple group of Lie type of rank $r$ over $\F_{p^f}$. Then there are
\begin{enumerate}
\item at most $2r+\omega(f)+10$ maximal subgroups of $G$ that contain $\soc(G)$
\item at most $100r$ conjugacy classes of maximal subgroups of $G$ that have type (I*), (III) or (V)
\item at most $(r+1)(\omega(f)+2)$ conjugacy classes of maximal subgroups of $G$ that have type (II).
\end{enumerate}
\end{proposition}

We establish some lemmas before proving Proposition~\ref{prop:count}.

\begin{lemma} \label{lem:count_goursat_1}
Let $H$ be a subgroup of $S_4$ and let $n$ be a positive integer. Then $G = H \times C_n$ has at most $9+\omega(n)$ maximal subgroups.
\end{lemma}

\begin{proof}
Write $\M(X)$ for the set of maximal subgroups of $X$ and $\mathrm{Hom}(X,Y)$ for the set of homomorphisms from $X$ to $Y$. Goursat's lemma (see \cite[(4.3.1)]{ref:Scott64}, for example) implies that 
\begin{equation}
|\M(G)| = |\M(H)|+|\M(C_n)|+\sum_{\text{prime $p \div n$}}(|\mathrm{Hom}(H,C_p)|-1). \label{eq:goursat}
\end{equation}
Note that $|\M(C_n)| = \omega(n)$ and $|\mathrm{Hom}(H,C_p)|=1$ unless $p \in \{2,3\}$. It remains to check that $|\M(H)|+|\mathrm{Hom}(H,C_2)|+|\mathrm{Hom}(H,C_3)| \leq 11$, which is easy to do for each $H \leq S_4$.
\end{proof}


\begin{lemma} \label{lem:count_goursat_2}
Let $m$ and $n$ be positive integers. Then
\begin{enumerate}
\item any semidirect product $C_m{:}C_n$ has at most $m+\omega(n)$ maximal subgroups
\item any semidirect product $C_m{:}(C_n \times C_2)$, where the generator of the $C_2$ subgroup inverts every element of the $C_m$ subgroup, has at most $2m+\omega(n)+2$ maximal subgroups.
\end{enumerate}
\end{lemma}

\begin{proof}
For part~(i), write $G = \< a, b \mid a^m, b^n, a^ba^{-k} \>$ and $H = \< b\> \cong C_n$, and for part~(ii) write $G = \< a, b,c \mid a^m, b^n, c^2, a^ba^{-k}, a^ca, [b,c] \>$ and $H = \< b, c\> \cong C_n \times C_2$ (in both cases, we assume that $\mathrm{gcd}(k,m)=1$ and $m \div (k^n-1)$). 

Let $M$ be a maximal subgroup of $G$, and write $A = M \cap \<a\>$. The possibilities for $M$ correspond to the maximal subgroups of $G/A$. If $A = \<a\>$, then $G/A = H$, so applying \eqref{eq:goursat} as in the proof of Lemma~\ref{lem:count_goursat_1}, we see that the number of possibilities for $M$ is at $\omega(n)$ and $\omega(n)+2$ in cases~(i) and~(ii), respectively. Now assume that $A < \<a\>$. First note that $M$ projects into $H$, for otherwise $M < \<a, M_0\>$ for a maximal subgroup $M_0$ of $H$, contradicting the maximality of $M$. We next claim that $|\<a\>:A|$ is prime. For a contradiction, suppose otherwise. Then $A < \< a^p\>$ for some prime divisor $p$ of $m$. Now $M < \< M, a^p \>$. If $\< M, a^p \> = G$, then $a \in \< M, a^p\> = \<a^p\>M$ which is impossible since $M \cap \< a \> = A \leq \<a^p\>$, so $M < \< M, a^p \> < G$, which contradicts the maximality of $M$. Therefore, $|\<a\>:A|$ is a prime divisor of $m$. To finish, we divide into the cases (i) and (ii).

For (i), if $|\<a\>:A| = p$, then $M = \< A, a^ib\>$ where $0 \leq i < p$, so there are at most $p$ possibilities for $M$. This means that if $m = p_1^{e_1} \dots p_k^{e_k}$, where $p_1, \dots, p_k$ are the distinct prime divisors of $m$, there are at most $p_1+\dots+p_k+\omega(n) \leq m+\omega(n)$ maximal subgroups of $G$.

For (ii), if $|\<a\>:A| = p$, then $M = \< A, a^ib, a^jc\>$ where $0 \leq i,j < p$. Now $[a^ib,a^jc] \in \< a\> \cap M = A$, but $[a^ib,a^jc] = a^{(1-k)j-2ik}$, so $Aa^{2ik} = Aa^{(1-k)j}$ and thus there are at most two choices for $i$ for each choice of $j$. Since there are at most $p$ choices for $j$, there are at most $2p$ choices for $M$. As in the previous case, if $m = p_1^{e_1} \dots p_k^{e_k}$, then there are at most $2p_1+\dots+2p_k+\omega(n)+2 \leq 2m+\omega(n)+2$ maximal subgroups of $G$.
\end{proof}

\begin{lemma} \label{lem:count_c7}
Let $G$ be an almost simple group with socle $\PSL^\e_n(q)$. Then $G$ has at most $n$ conjugacy classes of maximal $\C_7$ subgroups.
\end{lemma}

\begin{proof}
Let $t$ be the largest integer such that $n=s^t$ for some integer $s$. Then every maximal $\C_7$ subgroup of $G$ has type $\mathrm{GL}^\e_m(q) \wr S_k$ where $k$ divides $t$ and $n=m^k$ and by \cite[Tables~3.5.A \&~3.5.B]{ref:KleidmanLiebeck},  there are at most $n/m$ $G$-classes of subgroups of a given type. Therefore, the number of $G$-classes of $\C_7$ subgroups is at most $\sum_{k \mid t} s^{t-t/k} \leq \sum_{i=0}^{t-1} s^i = (s^t-1)/(s-1) < s^t = n$.
\end{proof}

\begin{proof}[Proof of Proposition~\ref{prop:count}]
First consider part~(i). For now assume that $\soc(G) \neq \PSL^\pm_n(q)$. Then $\Out(\soc(G)) = H \times C_{df}$ where $H \leq S_4$ and $d \leq 3$, see \cite[Table~2]{ref:BurnessGuralnickHarper21} for exceptional groups and \cite[Section~5.2]{ref:HarperLNM} for classical groups. Therefore, Lemma~\ref{lem:count_goursat_1} implies that $\Out(\soc(G))$ has at most $\omega(f)+10$ maximal subgroups. It remains to assume that $\soc(G) = \PSL^\pm_n(q)$. In this case, Lemma~\ref{lem:count_goursat_2} implies that $\Out(\PSL_n(q)) = C_{\mathrm{gcd}(q-1,n)}{:}(C_2 \times C_f)$ has at most $2n+\omega(f)+2 = 2r+\omega(f)+4$ maximal subgroups and $\Out(\PSU_n(q)) = C_{\mathrm{gcd}(q+1,n)}{:}C_{2f}$ has at most $n+\omega(2f) \leq r+\omega(f)+2$ maximal subgroups. This proves part (i).

For parts~(ii) and~(iii), an $\Inndiag(\soc(G))$-class yields at most $|\Inndiag(\soc(G)):\soc(G)|$ classes in $G$, and $|\Inndiag(\soc(G)):\soc(G)| \leq 4$ unless $\soc(G) = \PSL^\pm_n(q)$, in which case $|\Inndiag(\soc(G)):\soc(G)| \leq n = r+1$. Part~(iii) now follows by the observation that there are at most $\omega(f)+2$ classes in $\Inndiag(\soc(G))$. Part~(ii) is easily verified by consulting  \cite[Theorem~2]{ref:LiebeckSeitz90} for exceptional groups and \cite[Chapter~3]{ref:KleidmanLiebeck} and \cite{ref:Schaffer99} for classical groups (we use Lemma~\ref{lem:count_c7} in the one slightly more difficult case).
\end{proof}

\section{Proofs of the main theorems} \label{s:proof}

\subsection{Independent sets for primitive actions of almost simple groups} \label{ss:proof_height}

This section is devoted to proving Theorem~\ref{thm:height}. In the introduction, we defined height in terms of sequences of points, but clearly the ordering is irrelevant, so from now on we focus on sets of points. That is, for a group $G$ acting on a set $\Omega$, a subset $S \subseteq \Omega$ is \emph{independent} if $G_{(S')} > G_{(S)}$ for all proper subsets $S'$ of $S$, and the \emph{height}, denoted $H(G,\Omega)$, is the maximum size of an independent subset of $G$ on $\Omega$.

\begin{lemma} \label{lem:independent}
Let $G$ be a finite group acting on a set $\Omega$. Let $N$ be a normal subgroup of $G$ such that $G/N$ is cyclic. Then $H(G,\Omega) \leq H(N,\Omega) + \omega(|G/N|)$. 
\end{lemma}

\begin{proof}
Let $\Gamma \subseteq \Omega$ be an independent set for $G$ of size $H(G,\Omega)$. Fix $\Delta \subseteq \Gamma$ such that $\Delta$ is independent for $N$ and $N_{(\Delta)} = N_{(\Gamma)}$, so, in particular,  $|\Delta| \leq H(N,\Omega)$. This is always possible by \cite[Lemma~2.4]{ref:GillLodaSpiga}, but the argument is short so we give it: if $\Gamma$ is independent for $N$, then let $\Delta = \Gamma$; otherwise, there exists a proper subset $\Gamma' \subseteq \Gamma$ such that $N_{(\Gamma')} = N_{(\Gamma)}$, and we repeat the argument replacing $\Gamma$ with $\Gamma'$.

Let $\p\:G \to G/N$ be the quotient map, and let $p_1 < \dots < p_k$ be the prime divisors of $|G/N|$, so, in particular, $k = \omega(|G/N|)$. Write $|\p(G_{(\Gamma)})| = p_1^{e_1} \cdots p_k^{e_k}$. Fix $1 \leq i \leq k$. Suppose that for all $\alpha \in \Gamma \setminus \Delta$, the $p_i$-part of $|\p(G_{(\Delta \cup \{\alpha\})})|$ strictly exceeds $p_i^{e_i}$. Then since $G_{(\Gamma)} = \bigcap_{\alpha \in \Gamma \setminus \Delta} G_{(\Delta \cup \{\alpha\})}$ and $\p(G)$ is cyclic, the $p_i$-part of $|\p(G_{(\Gamma)})|$ strictly exceeds $p_i^{e_i}$, which is a contradiction. Therefore, there exists $\alpha_i \in \Gamma \setminus \Delta$ such that the $p_i$-part of $|\p(G_{(\Delta \cup \{\alpha_i\})})|$ is $p_i^{e_i}$. Thus $|\p(G_{(\Delta \cup \{\alpha_1, \dots, \alpha_k\})})| = p_1^{e_1} \cdots p_k^{e_k} = |\p(G_{(\Gamma)})|$. However, $N_{(\Delta \cup \{\alpha_1, \dots, \alpha_k\})} = N_{(\Gamma)}$, so we have $G_{(\Delta \cup \{\alpha_1, \dots, \alpha_k\})} = G_{(\Gamma)}$. Since $\Gamma$ is independent for $G$, we deduce that $\Gamma = \Delta \cup \{\alpha_1, \dots, \alpha_k\}$, which implies that $H(G,\Omega) = |\Gamma| \leq |\Delta| + k \leq H(N,\Omega) + \omega(f)$, as sought.
\end{proof}

\begin{proof}[Proof of Theorem~\ref{thm:height}]
Let $G$ be an almost simple group of Lie type of rank $r$ over $\F_{p^f}$, where $p$ is prime, acting primitively on $\Omega$. Let $G_0 = \soc(G)$, so $G_0 \leq G \leq \Aut(G_0)$. Now \cite[Theorem~2.5.12]{ref:GorensteinLyonsSolomon98} implies that $\Aut(G_0)$ has a normal subgroup $N$ such that $\Aut(G_0)/N = C_f$ and $|N/G_0| \leq 6r$. Since $G/(G \cap N) \cong GN/N$, by Lemma~\ref{lem:independent}, 
\[
H(G,\Omega) \leq H(G \cap N,\Omega) + \omega(|GN/N|) \leq H(G \cap N,\Omega) + \omega(f),
\] 
and, by \cite[Lemma~2.8]{ref:GillLodaSpiga}, 
\[
H(G \cap N,\Omega) \leq H(G_0,\Omega) + \ell((G \cap N)/G_0) \leq H(G_0,\Omega) + \ell(N/G_0).
\] 
Now $\ell(N/G_0) \leq \log_2{|N/G_0|} \leq \log_2(6r) \leq 3r^3$, so $H(G,\Omega) \leq H(G_0,\Omega) + 3r^3 + \omega(f)$. While the action of $G_0$ on $\Omega$ need not be primitive, as explained in the final paragraph of the proof of \cite[Corollary~3]{ref:GillLiebeck23}, we still have $H(G_0,\Omega) \leq 174r^8$, so $H(G_0,\Omega) \leq 177r^8 + \omega(f)$.
\end{proof}

\subsection{Minimal generating sets for almost simple groups of Lie type} \label{ss:proof_almost_simple}

We are now in a position to prove Theorem~\ref{thm:almost_simple}. 

\begin{proof}[Proof of Theorem~\ref{thm:almost_simple}]
By Theorem~\ref{thm:height} and Proposition~\ref{prop:length}, there exist constants $A,B,C > 0$ such that for all almost simple groups of Lie type $G$ of rank $r$ over $\F_{p^f}$, where $p$ is prime, the following both hold
\begin{enumerate}
\item if $G$ acts faithfully and primitively on a set $\Omega$, then $H(G,\Omega) \leq Ar^B+\omega(f)$
\item if $M$ is a maximal subgroup of $G$ of type (IV*), then $\ell(M) \leq C r$.
\end{enumerate}
Define $\alpha = \max\{ 100A, C \}$ and $\beta=B+2$ (note that $(\alpha,\beta) = (17700,10)$ is a valid choice here since $(A,B) = (177,8)$ and $C=192$ are valid for Theorem~\ref{thm:height} and Proposition~\ref{prop:length}).

Let $G$ be an almost simple group of Lie type of rank $r$ over $\F_{p^f}$ where $p$ is prime. Then we claim that 
\begin{equation}
m(G) \leq \alpha (r + \omega(f))^\beta. \label{eq:almost_simple}
\end{equation}
Let $X$ be a minimal generating set for $G$. For each $x \in X$, write $H_x = \< X \setminus \{x\} \>$ and let $M_x$ be a maximal subgroup of $G$ such that $H_x \leq M_x$. For distinct $x,y \in X$ note that $M_x \neq M_y$, for otherwise $\< X \setminus \{x\} \> \leq M_x$ and $\< X \setminus \{y\} \> \leq M_x$, so $G = \< X \> \leq M_x$, which is impossible.

First assume that for all $x \in X$ the maximal subgroup $M_x$ contains $\soc(G)$ or has type (I*), (II), (III) or (V). For a contradiction, suppose that $|X| > \alpha(r+\omega(f))^\beta$. This means that
\begin{gather*}
|X| > \alpha(r+\omega(f))^\beta 
 \geq 100A(r+\omega(f))^{B+2} 
 \geq (Ar^B + \omega(f)) \cdot 100(r+\omega(f))^2 \\
 \geq (Ar^B + \omega(f)) \cdot (100r + (r+1)(\omega(f)+2)) + (2r + \omega(f) + 10).
\end{gather*}
By Proposition~\ref{prop:count}(i), $\soc(G) \leq M_x$ for at most $2r + \omega(f) + 10$ elements $x$ of $X$. Therefore, $M_x$ is core-free for strictly greater than $(Ar^B + \omega(f)) \cdot (100r + (r+1)(\omega(f)+2))$ elements $x$ of $X$. Now Proposition~\ref{prop:count}(ii)--(iii) together with the pigeonhole principle implies that there exists a core-free maximal subgroup $M$ of $G$ and a subset $Y \subseteq X$ such that $|Y| > Ar^B + \omega(f)$ and for all $y \in Y$ there exists $g_y \in G$ such that $M_y = M^{g_y}$. 

We claim that $\bigcap_{y \in Y} M^{g_y} < \bigcap_{y \in Y \setminus \{y_0\}} M^{g_y}$ for all $y_0 \in Y$. To see this, it suffices to fix $y_0 \in Y$ and show that $\bigcap_{y \in Y \setminus \{y_0\}} M^{g_y} \not\leq M^{g_{y_0}}$. For a contradiction, suppose otherwise. First note that $\< X \setminus \{y_0\} \> = H_{y_0} \leq M_{y_0} = M^{g_{y_0}}$, Second note that for all $y \in Y \setminus \{y_0\}$ we have $y_0 \in \< X \setminus (Y \setminus \{y_0\}) \> \leq \< X \setminus \{y\} \> = H_y \leq M_y = M^{g_y}$, so $y_0 \in \bigcap_{y \in Y \setminus \{y_0\}} M^{g_y}$. Therefore, under the supposition that $\bigcap_{y \in Y \setminus \{y_0\}} M^{g_y} \leq M^{g_{y_0}}$, we deduce that $G = \< X \> \leq M^{g_{y_0}}$, which is absurd. This establishes the claim.

This means that $\{ Mg_y \mid y \in Y \}$ is an independent set for the action of $G$ on $G/M$, so Theorem~\ref{thm:height} implies that $|Y| \leq Ar^B + \omega(f)$ (see (ii) above), but this directly contradicts the fact that $|Y| > Ar^B + \omega(f)$. Therefore, we deduce that $|X| \leq \alpha(r+\omega(f))^\beta$.

Next assume that there exists $x \in X$ such that $M_x$ has type (IV*). We clearly have the inequalities $|X| \leq m(H_x) + 1 \leq \ell(H_x) + 1 \leq \ell(M_x) + 1$. Proposition~\ref{prop:length} implies that $\ell(M_x) \leq Cr \leq \alpha r$ (see (i) above), so $|X| \leq \alpha r+1 \leq \alpha(r+\omega(f))^\beta$.

We now pause to observe that we have proved \eqref{eq:almost_simple} when $\soc(G) = \PSL_2(p^f)$ since in this case the (I) coincides with (I*) and (IV) coincides with (IV*).

Having established the result in the base case where the rank $r$ is $1$, we now complete the proof by induction. Suppose that $r = s > 1$ and that \eqref{eq:almost_simple} holds for all groups with $r < s$. 

By Proposition~\ref{prop:maximal}, it remains to assume that $G$ is classical and $H_x$ is an almost simple group of Lie type defined over $\F_{p^e} \subseteq \F_{p^f}$ such that $\mathrm{rank}(M_x) < r$. Now, by induction,
\[
|X| \leq m(H_x) + 1 \leq \alpha (\mathrm{rank}(H_x) + \omega(e))^\beta + 1 \leq \alpha(r-1+\omega(f))^\beta + 1 \leq \alpha(r+\omega(f))^\beta.
\]
Therefore, in all cases $|X| \leq \alpha(r+\omega(f))^\beta$, as desired.
\end{proof}

We next show that, up to improving $\alpha$ and $\beta$, the bound in Theorem~\ref{thm:almost_simple} is best possible.

\begin{proposition} \label{prop:lower}
Let $G$ be a finite simple group of Lie type of rank $r$ over $\F_{p^f}$, where $p$ is prime. Then $m(G) \geq 2r + \omega(f)$.
\end{proposition}

\begin{proof}
Let $B$ be a Borel subgroup of $G$ containing a maximal torus $T$, let $\Phi$ be the corresponding root system of $G$ and let $\Delta = \{ \alpha_1, \dots, \alpha_r \}$ be a set of simple roots. We will now construct a minimal generating set for $G$. To refer to elements of $G$, we will use the standard Lie theoretic notation $x_\alpha(t)$ and $h_\alpha(t)$, see  \cite[Theorem~1.12.1]{ref:GorensteinLyonsSolomon98}, for example. 

For $1 \leq i \leq r$, let $x_i$ and $y_i$ be the root elements  $x_{\alpha_i}(1)$ and $x_{-\alpha_i}(1)$, respectively. Write $f = e_1^{a_1} \cdots e_k^{a_k}$ where $e_1, \dots, e_k$ are the distinct prime divisors of $f$ (so $k = \omega(f)$). For $1 \leq i \leq k$, let $f_i = e_i^{a_i}$, let $\l_i$ be a primitive element of the subfield $\F_{p^{f_i}}$ and let $z_i = h_{\alpha_1}(\l_i)$. 

We claim that $X = \{ x_1, \dots, x_r, y_1, \dots, y_r, z_1, \dots, z_k \}$ is a minimal generating set for $G$ (since $|X| = 2r+\omega(f)$ this establishes the result). The fact that $X$ generates $G$ follows from \cite[Theorem~1.12.1]{ref:GorensteinLyonsSolomon98} and \cite[Corollary~24.2]{ref:MalleTesterman11}. To see that $X$ is minimal, note that for all $1 \leq i \leq k$, the set $X \setminus \{z_i\}$ is contained in the subfield subgroup defined over the subfield $\F_{p^{f/e_i}}$, and for all $1 \leq i \leq r$, both of the sets $X \setminus \{x_i\}$ and $X \setminus \{y_i\}$ are contained in parabolic subgroups of type $P_i$ (corresponding to deleting node $i$ from the Dynkin diagram of $\Phi$).
\end{proof}

\begin{example} \label{ex:lower}
To elucidate the proof of Proposition~\ref{prop:lower}, let us give an explicit description of the minimal generating set when $G = \PSL_3(q)$ and $q = p^{f_1f_2}$ for distinct primes $f_1$ and $f_2$. Let $\l_1$ and $\l_2$ be primitive elements of $\F_{p^{f_1}}$ and $\F_{p^{f_2}}$ respectively. Then we obtain a minimal generating set $\{x_1,x_2,y_1,y_2,z_1,z_2\}$ where
\begin{gather*}
x_1 = \left( \begin{array}{ccc} 1 & 1 & 0 \\ 0 & 1 & 0 \\ 0 & 0 & 1 \end{array} \right), \quad
x_2 = \left( \begin{array}{ccc} 1 & 0 & 0 \\ 0 & 1 & 1 \\ 0 & 0 & 1 \end{array} \right), \quad
y_1 = \left( \begin{array}{ccc} 1 & 0 & 0 \\ 1 & 1 & 0 \\ 0 & 0 & 1 \end{array} \right), \quad
y_2 = \left( \begin{array}{ccc} 1 & 0 & 0 \\ 0 & 1 & 0 \\ 0 & 1 & 1 \end{array} \right) \\[5pt]
z_1 = \left( \begin{array}{ccc} \l_1 & 0 & 0 \\ 0 & \l_1^{-1} & 0 \\ 0 & 0 & 1 \end{array} \right), \quad
z_2 = \left( \begin{array}{ccc} \l_2 & 0 & 0 \\ 0 & \l_2^{-1} & 0 \\ 0 & 0 & 1 \end{array} \right).
\end{gather*}
\end{example}

\subsection{Minimal generating sets for an arbitrary finite group} \label{ss:proof_main}

We now use Theorem~\ref{thm:almost_simple} to prove Theorem~\ref{thm:main}. We first require the following reduction theorem, which was proved by Lucchini, Moscatiello and Spiga \cite[Theorem~1.4]{ref:LucchiniMoscatielloSpiga21}.

\begin{theorem} \label{thm:reduction}
Let $a \geq 1$ and $b \geq 2$. Let $G$ be a finite group. Assume that every composition factor $S_0$ of $G$ and every almost simple group $S$ with socle $S_0$ satisfies $m(S)-m(S/S_0) \leq a \cdot \omega(|S_0|)^b$. Then $m(G) \leq a \cdot \delta(G)^b$.
\end{theorem}

We have focussed on almost simple groups of Lie type since otherwise the required result follows from existing work in the literature as the following theorem highlights (this is noted in \cite[Lemma~4.5]{ref:LucchiniMoscatielloSpiga21} without an explicit constant).

\begin{theorem} \label{thm:non_lie}
There exists a constant $\gamma > 0$ such that if $G$ is an almost simple group that is not a group of Lie type, then $m(G) \leq \gamma \cdot \omega(|\soc(G)|)^2$. Moreover, this is true with $\gamma = 52$.
\end{theorem}

\begin{proof}
Let $\gamma = 52$. If $\soc(G)$ is sporadic, then $m(G) \leq \ell(G) \leq 52$ (see \cite[Tables~III and~IV]{ref:CameronSolomonTurull89}), and it is easy to check that the same bound holds when $\soc(G)$ is $A_6$ or ${}^2F_4(2)'$.

We can now assume that $\soc(G) = A_n$ for $n \neq 6$.  In this case, $G$ is $A_n$ or $S_n$ and Whiston proved that $m(G)$ is $n-2$ or $n-1$, respectively \cite{ref:Whiston00}. If $n \leq 53$, then $m(G) \leq n-1 \leq 52$. Otherwise, by \cite[Corollary~1]{ref:RosserSchoenfeld62} we know that $\pi(n) > n/\log{n}$, where $\pi$ is the prime-counting function and $\log$ is the natural logarithm. Noting that $\log{n} < \sqrt{n}$, these bounds give 
\[
m(G) \leq n-1 < (n/\log{n})^2 < \pi(n)^2 = \omega(|\soc(G)|)^2. \qedhere
\]
\end{proof}

The following lemma relates Theorems~\ref{thm:almost_simple} and~\ref{thm:reduction} for groups of Lie type.

\begin{lemma} \label{lem:omega}
Let $G$ be an almost simple group of Lie type of rank $r$ over $\F_{p^f}$ (where $p$ is prime). Then 
\[
\omega(|\soc(G)|) \geq \max(1,\tfrac{1}{2}(r-1))+\omega(f).
\]
\end{lemma}

\begin{proof}
It is easy to check that $|\soc(G)|$ is divisible by $p^{fd_1}-1$, $p^{fd_2}-1$, \dots $p^{fd_k} -1$ for some $d_1 < d_2 < \dots < d_k$ with $k \geq \max(1,\frac{1}{2}(r-1))$. Moreover, if $e_1, \dots, e_l$ are the distinct prime divisors of $f$, then $|\soc(G)|$ is divisible by $p-1$, $(p^{e_1}-1)/(p-1)$, $(p^{e_1 e_2}-1)/(p-1)$ \dots, $(p^{e_1 e_2 \dots e_l}-1)/(p-1)$. Note that $1 < e_1 < e_1e_2 < \dots < e_1e_2 \dots e_l < fd_2 < fd_3 \dots < fd_k$. By Zsigmondy's theorem \cite{ref:Zsigmondy82}, for all but at most one $i \in \{1, e_1, \dots, e_1e_2 \dots e_l, fd_2, fd_3, \dots, fd_k \}$ we may fix a primitive prime divisor of $p^i-1$. Noting that $|\soc(G)|$ is also divisible by $p$, we deduce that $\omega(|\soc(G)|) \geq \max(1,\frac{1}{2}(r-1)) + \omega(f)$. 
\end{proof}

We can now prove Theorem~\ref{thm:main}.

\begin{proof}[Proof of Theorem~\ref{thm:main}]
Let $\alpha, \beta, \gamma > 0$ be constants satisfying Theorems~\ref{thm:almost_simple} and~\ref{thm:non_lie}. Let us define $a = \max\{\alpha \cdot 3^\beta, \gamma\}$ and $b = \max\{\beta,2\}$ (note that $b=10$ is a valid choice here since $\beta=10$ is a valid choice in Theorem~\ref{thm:almost_simple}, and using $\alpha=10^5$ and $\gamma=52$ gives $a < 10^{10}$).

Let $G$ be a finite group, let $S_0$ be a composition factor of $G$ and let $S$ be an almost simple group with socle $S_0$. First assume that $S_0$ is alternating, sporadic or the Tits group. Then Theorem~\ref{thm:non_lie} implies that $m(S) \leq \gamma \cdot \omega(|S_0|)^2$. Now assume that $S_0$ is finite simple group of Lie type of rank $r$ over $\F_{p^f}$, where $p$ is prime. Then Theorem~\ref{thm:almost_simple} gives us the bound $m(S) \leq \alpha(r+\omega(f))^\beta$. By Lemma~\ref{lem:omega}, $\omega(|S_0|) \geq \max(1,\frac{1}{2}(r-1))+\omega(f) \geq \frac{1}{3}(r+\omega(f))$, so $m(S) \leq \alpha \cdot 3^\beta \cdot \omega(|S_0|)^\beta$. Therefore, in both cases, $m(S) \leq a \cdot \omega(|S_0|)^b$. By Theorem~\ref{thm:reduction}, this establishes that $m(G) \leq a \cdot \delta(G)^b$. 
\end{proof}

\vspace{2pt}

\noindent Scott Harper \newline
School of Mathematics and Statistics, University of St Andrews, KY16 9SS, UK \newline
\texttt{scott.harper@st-andrews.ac.uk}

\end{document}